\documentclass[12pt]{amsart}
\usepackage{geometry} % See geometry.pdf to learn the layout options. There are lots.
\usepackage{a4wide}
\usepackage{amsmath,amsfonts,amsthm,amssymb,amsopn,cite,mathrsfs}
\usepackage{%epsfig,
verbatim}
\usepackage{subfigure}
\usepackage{graphicx}
\newcommand{\beq}{\begin{equation}}
\newcommand{\eeq}{\end{equation}}
\newcommand{\beqs}{\begin{equation*}}
\newcommand{\eeqs}{\end{equation*}}
\newcommand{\beqq}{\begin{equation}}
\newcommand{\eeqq}{\end{equation}}

\newcommand{\beqas}{\begin{eqnarray*}}
\newcommand{\eeqas}{\end{eqnarray*}}
\newcommand{\beqa}{\begin{eqnarray}}
\newcommand{\eeqa}{\end{eqnarray}}

\newcommand{\bmuls}{\begin{multline*}}
\newcommand{\emuls}{\end{multline*}}
\newcommand{\bmul}{\begin{multline}}
\newcommand{\emul}{\end{multline}}

\def\bR{\mathbb{R}}
\def\bZ{\mathbb{Z}}

\def\C{\mathbb{C}}

\def\cG{\mathcal{G}}

\def\cF{\mathcal{F}}

\newtheorem{thm}{Theorem}
\newtheorem*{lemma*}{Lemma}
\newtheorem{lemma}{Lemma}

\newcommand{\LL}{L^2(\bR)}

\title{Phase space distribution of Gabor expansions}
\author{Gerard Ascensi}
\address{Faculty of Mathematics, University of Vienna, Nordbergstrasse 15, 1090 Vienna, Austria}
\email{gerard.ascensi@univie.ac.at}

\author{Yurii Lyubarskii}
\address{Department of Mathematical Sciences, Norwegian University of Science and Technology,
NO-7491 Trondheim, Norway} \email{yura@math.ntnu.no}

\author{Kristian Seip}
\address{Department of Mathematical Sciences, Norwegian University of Science and Technology,
NO-7491 Trondheim, Norway} \email{seip@math.ntnu.no}
\thanks{Lyubarskii and Seip are supported by the Research Council of Norway grant
10323200. Ascensi is supported by Fundaci\'o Ferran Sunyer i
Balaguer. This research is part of the European Science Foundation
Networking Programme ``Harmonic and Complex Analysis and Its
Applications'' (HCAA). It was done during a visit of Ascensi to
the Norwegian University of Science and Technology (NTNU). He
acknowledges the support and hospitality of the Department of
Mathematical Sciences at NTNU}
\begin{document}

\maketitle

\begin{abstract}
We present an example of a complete and minimal Gabor system
consisting of time-frequency shifts of a Gaussian, localized at
the coordinate axes in the time-frequency plane (phase space).
Asymptotically, the number of time-frequency shifts contained in a
disk centered at the origin is only $2/\pi$ times the number of
points from the von Neumann lattice found in the same disk.
Requiring a certain regular distribution in phase space, we show
that our system has minimal density among all complete and minimal
systems of time-frequency shifts of a Gaussian.
\end{abstract}

\section{Introduction}

For a function $g$ in $\LL$ and two real numbers $x$ and $y$, we
define
\beq
\label{shift} \rho_{x,y}g(t)= e^{2i\pi y t}g(t-x).
\eeq
We refer to $\rho_{x,y}g$ as the time-frequency shift of $g$ with
respect to the point $(x,y)$ in phase space. Given $g$ in $\LL$ and
a sequence $\Lambda$ of distinct points in $\bR^2$, we are
interested in the spanning properties of the Gabor system \beq
\label{gabor}
\cG_{g,\Lambda}=\{\rho_{x,y}g:\ (x,y)\in \Lambda\}.
\eeq It is well known (see \cite{Jan94}, \cite{RS95},
\cite{CDH99}) that for $\cG_{g,\Lambda}$ to be a frame for $\LL$,
the lower Beurling--Landau density of $\Lambda$ has to be at least
$1$. This means that the minimal number of points from $\Lambda$
to be found in a disk of radius $r$ must be at least $\pi
r^2+o(r^2)$ when $r\to\infty$. The canonical case is when
$\Lambda$ is a lattice $a\bZ\times b\bZ$, the lower
Beurling--Landau density of which is $1/(ab)$ for positive lattice
constants $a$ and $b$. When $g$ is a Gaussian, the lattice
structure of $\Lambda$ is inessential and we have the following
complete description: Assuming $\Lambda$ is a separated sequence,
$\cG_{g,\Lambda}$ is a frame for $\LL$ if and only if the
Beurling-Landau density of $\Lambda$ exceeds 1 \cite{Lyu, Se1,
Se2}.

It is easy to construct examples of complete Gabor systems
$\cG_{g,\Lambda}$ for which the Beurling--Landau density of
$\Lambda$ is $0$. However, when $\Lambda=a\bZ\times b\bZ$, the
density condition $ab\leq 1$ is still necessary for
$\cG_{g,\Lambda}$ to be complete in $\LL$ (see \cite{Rie81}), and it
seems to be commonly thought that $\Lambda$ should somehow be
uniformly localized throughout the entire phase space $\bR^2$ for
there to be nice expansions associated with $\cG_{g,\Lambda}$. The
main point of the present note is to show that the phase space
distribution of $\Lambda$ may be dramatically different from that of
a lattice if we require $\cG_{g,\Lambda}$ to be a complete and
minimal system in $\LL$. Indeed, we have the following example when
$g$ is a Gaussian:

\begin{thm}\label{theorem1}
Set $g(t)=\exp(-\pi t^2)$ and
\[
\Lambda=\{(-1,0),(1,0)\} \bigcup \left\{(0,\pm \sqrt{2n}):\ n=1,2,3,...\right\}
\bigcup \left\{(\pm \sqrt{2n},0):\ n=1,2,3,...\right\}.
\]
Then
$\cG_{g,\Lambda}$ is a complete and minimal system in $\LL$.
\end{thm}

We see that the sequence $\Lambda$ of this example is localized at
the two coordinate axes in $\bR^2$. In view of the density results
mentioned above, it may seem surprising that there exist arbitrarily
large disks containing no points from $\Lambda$. It is also quite
striking that the number of points $\lambda=(\xi,\eta)$ from
$\Lambda$ satisfying $\xi^2+\eta^2\le r^2$ is $2r^2+O(1)$ when $r\to
\infty$. This is asymptotically only $2/\pi$ times the number of
points from a lattice of Beurling--Landau density 1 to be found in
any disk of radius $r$. It is well-known that when $\Lambda$ is a
lattice of Beurling--Landau density $1$ with one point omitted,
$\cG_{g,\Lambda}$ is also a complete and minimal system, see e.g.
\cite{LSe99}. Thus the phase space %densities
distribution of complete and minimal Gabor systems may differ in a
rather fundamental way.

Our second theorem shows that we have in fact identified two extreme
configurations among all complete and minimal systems of
time-frequency shifts of a Gaussian, assuming the time-frequency
shifts to be regularly distributed as defined in the theory of
entire functions. To make this statement precise, we switch to
complex notation, i.e., we view our sequence $\Lambda$ as a subset
of the complex plane\footnote{In what follows, we will allow
ourselves to view the sequence $\Lambda$ alternately as a subset of
$\bR^2$ and of $\C$, with the tacit choice of viewpoint depending on
whatever fits the context.} via the map $(\xi,\eta)\mapsto
\xi+i\eta$. Given $r>0$ and $0\le \theta < \vartheta \leq 2\pi$, we
write $n_{\Lambda}(r, \theta, \vartheta)= \# \{\lambda \in \Lambda;
|\lambda|< r, \ \theta < \text{arg}\lambda \leq \vartheta\}$, and
we say that $\Lambda$ has angular density if the limit
\[
\Delta_\Lambda(\theta,\vartheta)= \lim_{r\to \infty}
\frac{n_{\Lambda}(r, \theta, \vartheta)}{\pi r^2}
\]
exists for all $\theta$ and $\vartheta$, except possibly for a
countable set of such pairs of numbers. From this definition we see
that if $\Lambda$ has angular density, then the limit
$\Delta_\Lambda(0,2\pi)$ exists. We set
$\Delta_\Lambda=\Delta_\Lambda(0,2\pi)$ and refer to this number as
the density of $\Lambda$. If $\Lambda$ has angular density and, in
addition, the limit
\[ \lim_{r\to\infty} \sum_{\lambda\in \Lambda,\
|\lambda|<r}\lambda^{-2} \] exists, we say that $\Lambda$ is a
regularly distributed set. We refer to Chapter~II of \cite{Le56} for
more refined versions of these definitions (depending on the order
of the functions in question) and their significance in the theory
of entire functions of completely regular growth.

It is clear that the sequence $\Lambda$ of Theorem~\ref{theorem1} as
well as any lattice $a\bZ\times b\bZ$ is a regularly distributed
set. The following theorem therefore places our examples in context.

%Indeed,
%letting $n_\Lambda(r)$ be the number of points
%$\lambda=(\xi,\eta)$ from $\Lambda$ satisfying $\xi^2+\eta^2\le
%r^2$ and defining the upper and lower densities (see e.g. \cite{Le96}, Lecture 3)
%as
%\[
%\Delta^-(\Lambda)=\liminf_{r\to\infty}\frac{n_\Lambda(r)}{\pi r^2}
%\ \ \ \text{and} \ \ \
%\Delta^+(\Lambda)=\limsup_{r\to\infty}\frac{n_\Lambda(r)}{\pi
%r^2},
%\]

\begin{thm}\label{theorem2}
Set $g(t)=\exp(-\pi t^2)$ and let $\Lambda$ be a regularly
distributed set. If $\cG_{g,\Lambda}$ is a complete and minimal
system in $\LL$, then $2/ \pi \leq \Delta_\Lambda \leq 1.$

%$\Delta^-(\Lambda)\ge 2/\pi$ and
%$\Delta^+(\Lambda)\le 1$.
\end{thm}

There are obvious ways in which our two theorems can be generalized:
The example of Theorem~1 can be extended to much more general sets,
and using the notion of proximate order as defined in Chapter 1 of
\cite{Le56}, one may relax the assumption about regular distribution
of $\Lambda$ in Theorem~\ref{theorem2}. We decided not to pursue
such elaborations; our aim is to present the essence of an
interesting qualitative phenomenon, avoiding as far as possible
obscuring technicalities.

The assumption that $\cG_{g,\Lambda}$ be complete and minimal
ensures the existence of a unique biorthogonal system in $\LL$. Thus
we may write down meaningful Fourier expansions for functions in
$L^2(\bR)$ with respect to the system $\cG_{g,\Lambda}$. However,
typically this expansion is not convergent, so that one may have to
find a special procedure to recover a function from its Fourier
coefficients.
%in this more general setting, one may typically have to resort to
%alternate reconstruction methods in order to recover the expanded
%function.
In \cite{LSe99}, a simple summation method was shown to do the job
for certain sequences similar to lattices. It would be interesting
to identify those functions for which the Gabor expansion exhibited
in Theorem~\ref{theorem1} converges, and to find an appropriate
summation method for arbitrary functions in $\LL$.

We can find examples of complete sets $\cG_{g,\Lambda}$ with
$\Lambda$ localized at a single line \cite{Zal78, Ole97}. It has
been proved that there are no Schauder bases of Gabor systems with
such phase space localization \cite{Heil}, but we still do not know
if there exists such a system $\cG_{g,\Lambda}$ that is complete and
minimal in $\LL$. In the case of Gaussian functions, the
corresponding sequence $\Lambda$ would have to be a regularly
distributed set, but then its density would be $0$, and this is
incompatible with Theorem~\ref{theorem2}. It has also been proved
that there are no complete and minimal systems of time-frequency
shifts of the Poisson function \cite{BrM07}.

The contrast with what is known about Paley--Wiener spaces and
families of complex exponentials in $L^2$ of an interval is worth
noting. In this case, we have the following precise density
condition: If the family $\{\exp(i\lambda t)\}_{\lambda\in \Lambda}$
is a complete and minimal system in $L^2(-a,a)$, then the number of
points from $\Lambda$ to be found in the disk $|z|\le r$ behaves
asymptotically as $2ar/\pi$ when $r\to\infty$; see Theorem~1 in
Lecture 17 of \cite{Le96}.

In the next brief section, we will make a transition to the Fock
space of entire functions, and then the remaining two sections are
devoted to proving our two theorems.

\section{Transition to the Fock space}

For the rest of the paper, $g$ denotes the normalized Gaussian:
$g(t)=2^{1/4}e^{-\pi t ^2}$. We define the Bargmann transform in
the following way: \beq \label{bt} \mathcal B f(z)=e^{-i\pi
xy}e^{\frac \pi 2 |z|^2}\int_{-\infty}^\infty f(t)
\overline{(\rho_{x,-y}g)(t)}dt=
2^{\frac 14 }\int_{-\infty}^\infty f(t) e^{-\pi t^2}e^{2\pi t z}
e^{-\frac \pi 2 z^2}dt,
\eeq
%Our proofs make essential use of the Bargmann transform, which is
%defined in the following way:
%\[
%\mathcal B f(z)=2^{\frac{1}{4}}\int_{-\infty}^\infty f(t)e^{-\pi\left(
%(t-x)^2-\frac{x^2}{2}\right)} \,dt,\quad z=x+iy.
%\]
%T
where we have put $z=x+iy$. The Bargmann transform maps $\LL$
isometrically onto the Fock space
\begin{equation*}
\mathcal{F}=\left\{ F : F \text{ is entire and }
\|F\|_{\mathcal{F}}^2=\int_{\C}|F(z)|^2e^{-\pi|z|^2}\,dm(z)<\infty\right\};
\end{equation*}
here $dm$ denotes planar Lebesgue measure. See \cite{Charly,
Folland} for this fact and other basic properties of the Bargmann
transform. Employing methods from the theory of entire functions,
we may prove much stronger results for a Gaussian than what can be
obtained for arbitrary functions in $\LL$.

% We will begin by proving a simple lemma which translate our
%problems into questions about sets of uniqueness for the Fock
%space $\mathcal{F}$.
We recall that a sequence of distinct points $\Lambda$ in $\C$ is a
set of uniqueness for $\mathcal{F}$ if the only function in
$\mathcal{F}$ vanishing at every point from $\Lambda$ is the zero
function. A set of uniqueness $\Lambda $ has zero excess if it fails
to be a set of uniqueness for $\mathcal{F}$ on the removal of any
one of the points from $\Lambda$. Note that since a function $F$ in
$\mathcal{F}$ remains in $\mathcal{F}$ when we change the position
of a single zero of $F$, it does not matter which particular point
we remove from $\Lambda$. A standard duality argument shows that
\eqref{bt} yields the following translation of our problem.

\begin{lemma}
%Set $g(t)=\exp(-\pi t^2)$ and let $\Lambda=\{ (\xi_j,\eta_j)\}$ be
%a sequence of distinct points in $\bR^2$.
The system $\cG_{g,\Lambda}$
is complete and minimal %system
in $\LL$ if and only if
$\Lambda$ is a set of uniqueness of zero excess
for $\mathcal{F}$.
\end{lemma}
\begin{proof}
We set $\overline{\Lambda}=\{ (\xi_j,-\eta_j)\}$ and note that, by
symmetry, $\cG_{g,\Lambda}$ is a complete and minimal system in
$\LL$ if and only if $\cG_{g,\overline{\Lambda}}$ is a complete
and minimal system in $\LL$.
%
%We observe that, with $z=x+iy$,
%\[ Bf(z)=e^{-\pi ixy}e^{\frac{\pi}{2}|z|^2} \int_{-\infty}^\infty
%f(t)\overline{(\rho_{x,-y}g)(t)}\,dt\] for every $f$ in $\LL$.
Since every function in $\mathcal{F}$ can be expressed as the
Bargmann transform of some function in $\LL$
%in this way
and the factor $e^{-\pi ixy}e^{\frac{\pi}{2}|z|^2}\neq 0$ at every
point $z$, the result follows by duality.
\end{proof}

\section{Proof of Theorem~\ref{theorem1}}
%\noindent {\em Proof of the theorem, uniqueness.}
This proof will use some standard notions and results from the
theory of entire functions of exponential type, such as indicator
diagrams, indicator functions, and the Phragm\'{e}n--Lindel\"{o}f
principle. Good references for this material are Lectures 6 and 9 in
\cite{Le96}, Chapter~1 of \cite{Le56}, and Chapter~5 of \cite{Bo54}.

In this proof, the meaning of the notation $A(w)\asymp B(w)$ for $w$
in some set $W$ is that the ratio between the two positive functions
$A(w)$ and $B(w)$ is bounded from below and above by positive
constants independent of $w$ in $W$.

To prove that $\Lambda$ is a set of uniqueness for $\mathcal{F}$,
we will argue by contradiction. Thus we begin by assuming that
there does indeed exist a nontrivial function $\Phi$ in $\cF$
vanishing at $\Lambda$. We may assume that $\Phi$ is an even
function. Were it not, we could replace it by $\Phi(z)+\Phi(-z)$,
which is again a function in $\cF$ vanishing on $\Lambda$;
admittedly, this function is identically $0$ if $\Phi$ is an odd
function, but were this the case, we could replace $\Phi$ by
$\Phi(z)/z$.

We note that $\Lambda$ is the zero set of the function
\[s(z)= (z^2-1) z^{-2} \sin \frac {\pi z^2}{2}.\] It follows that
the function $\Psi(z)=\Phi(z)s(z)^{-1}$
%such that
% \beq \label{zero_function} \Psi(z)=\Phi(z)s(z)^{-1}
% \eeq
is also an even entire function of order at most 2.

%We will use the standard techniques of theory of entire functions
%(see e.g. \cite{Le96}, and also the following property of the Fock space:
%(see e.g. \cite{LSe99}) {\em I have to check whether it is really written there or
%elsewhere}.

We will use the following elementary estimate for the sine
function:
\[
|\sin \pi z | \asymp \exp (\pi |\Im z|), \ \ \ \text{dist}(z, \bZ)
>\varepsilon,
\]
where $\varepsilon$ is an arbitrary positive number. It follows from
this estimate that \beq \label{gen_estimate} |s(re^{i\theta})|\asymp
e^{\frac{\pi}{2} r^2 |\sin 2\theta|}, \ \ \
\text{dist}(re^{i\theta}, \Lambda) > \epsilon r^{-\frac{1}{2}}.
\eeq
In particular, %\beq\label{first_estimate}
\[ |s(re^{i(\frac{\pi}{4} + k \frac{\pi}{2})})| \asymp
e^{\frac{\pi}{2} r^2}, \ \ \ k=0,1,2,3 \ \ \text{and} \ \ r>0. %\eeq
\]
%\beq
%\label{second_estimate}
%|s(re^{ik \frac \pi 2})| \asymp 1, \
%\text{dist}(re^{ik\frac \pi 2}, \Lambda)> \epsilon r^{\frac 12} \ k=0,1,2,3.
%\eeq
Combining the latter estimate with the
% trivial
fact that every function $F$ in $\cF$ satisfies \beq
\label{inequality} |F(z)|= o(e^{\frac{\pi}{2}|z|^2}), \ \ z \to
\infty \eeq (see Lemma 3 in \cite{LSe99} for a stronger statement),
we obtain
%Relations \eqref{inequality} and
%\eqref{first_estimate} yield
\begin{equation}
\label{cross_decay} |\Psi(re^{i(\frac{\pi}{4} + k \frac \pi 2)}) |
= o(1) \ \ \text{as}\ \ r\to \infty, \ k=0,1,2,3.
\end{equation}

Since $\Psi$ is an even function, %\beq\label{new_function}
$\Omega(z) = \Psi(e^{i\frac{\pi}{4}}z^{\frac{1}{2}})$
%\eeq
is an entire function of exponential type. We observe that
\eqref{cross_decay} yields \beq \label{line_decay}
|\Omega(x)|=o(1)
\eeq
when $x \to \pm \infty$. Hence
%, for an appropriate $\tau\geq 0$,
the indicator diagram of $\Omega$ is the segment
$[-i\tau_-,i\tau_+]$ on the imaginary axis, and its indicator
function is \beq \label{om_indicator} h_\Omega(\theta) =
\limsup_{r\to \infty} r^{-1} \log |\Omega(re^{i\theta})|=
\begin{cases} \tau_+ \sin \theta & \text{if} \ \ \sin \theta >0, \\
-\tau_-\sin \theta & \text{if} \ \ \sin \theta <0.
\end{cases}
\eeq If $\max\{\tau_+,\tau_-\}=0$, then the
Phragm\'{e}n--Lindel\"{o}f principle and \eqref{line_decay} show
that $\Omega$ is a bounded entire function and therefore a constant,
by Liouville's theorem. Thus $\Omega \equiv 0$ and, consequently,
$\Phi \equiv 0$, which is the desired contradiction.

It remains to consider the possibility of $\max\{\tau_+,\tau_-\}>0$.
By symmetry, we may confine ourselves to the case $\tau_+>0$.
%We consider next the case $\tau>0$.
We assume that $\sin \theta >0$ and rewrite
\eqref{om_indicator} as %\beq \label{psi_indicator}
\[
h_\Psi (\theta)= \limsup_{r\to \infty} r^{-2} \log |\Psi(re^{i\theta})|=
\tau_+ |\sin (2\theta + \pi/2 )|= \tau_+ |\cos (2\theta)|.
\]
%\eeq
Combining this with \eqref{gen_estimate} and using that
$\Psi(z)=\Phi(z)s(z)^{-1}$, we get
\beq
\label{onehand} \limsup_{r\to\infty} r^{-2} \log |\Phi(re^{i\theta})|=
\frac \pi 2 |\sin 2 \theta|+ \tau_+ |\cos (2\theta)|.
\eeq
On the other hand,
\eqref{inequality} with $F=\Phi$ yields
\beq
\label{otherhand}
\limsup_{r\to \infty} r^{-2} \log |\Phi(re^{i\theta})|\leq \frac \pi 2.
\eeq
The two relations \eqref{onehand} and
\eqref{otherhand} are incompatible because an elementary calculus
argument shows that, for each $\tau_+>0$, there exists $\delta >0$
such that
\[
\frac \pi 2 |\sin 2 \theta|+ \tau |\cos (2\theta)| > \frac \pi 2 \]
for $0<|\theta- \pi/4 | < \delta$. Thus we cannot have
$\max\{\tau_+,\tau_-\}>0$ and conclude that $\Lambda$ is a set of
uniqueness for $\cF$.

Direct estimates based on inequality \eqref{gen_estimate} show that
$(z-1)^{-1}s(z)$ is in $\cF$. In other words, $\Lambda\setminus
\{(1,0)\}$ is not a set of uniqueness for $\cF$, and consequently
$\Lambda$ has zero excess.

%\texttt{here should stay a remark, saying that the set $\Lambda$ in Theorem 1 is just
%a simplest representative of a much wider class of sets: evenness of $\Lambda$,
%is irrelevant, boundedness of $\sin$ on the reals may be replaced by $A_2$ etc.
%I have not decided yet on how this remark should be frased. }

\section{Proof of Theorem \ref{theorem2}}
%%%%theorem2

We will now assume that $\Lambda$ is a set of uniqueness of zero
excess for $\cF$. We fix a point $\lambda'$ in $\Lambda$ and let
$\Phi$ be a function $\cF$ with only simple zeros and
whose zero set coincides with $\Lambda \setminus \{\lambda'\}$. The densities of $\Lambda$
and $\Lambda \setminus \{\lambda'\}$ are of course the same.

We begin by noting that the right inequality $\Delta_\Lambda\le 1$
is a simple consequence of Jensen's formula; in particular, we only
need the required regular distribution of $\Lambda$ to define the
density $\Delta_\Lambda$. Set $n_\Lambda(r)=n_\Lambda(r,0,2\pi)$ and
assume that $\Phi(0)\neq 0$. Then Jensen's formula gives
\[
\int_0^r\frac{n_\Lambda(t)}{t}\,dt=\frac{1}{2\pi}\int_0^{2\pi}\log
|\Phi(re^{i\theta})|\,d\theta-\log|\Phi(0)|.
\]
Using again \eqref{inequality}, we obtain
\[ \liminf_{r\to\infty} \frac{n_\Lambda(r)}{\pi r^2}\le 1, \]
where the left-hand side equals the density $\Delta_\Lambda$
whenever the latter quantity exists.

The left inequality $2/\pi \le \Delta_\Lambda$ is an immediate
consequence of the following two lemmas on the indicator function
\[
h_\Phi(\theta)= \lim\sup_{r\to \infty} r^{-2} \log |\Phi(re^{i\theta})| , \ \ \ \theta \in [0, 2\pi].
\]
\begin{lemma}
\label{indicator1} Let $\Phi$ be a nontrivial entire function whose
zero set $\Lambda$ is a regularly distributed set. Then
\[ \Delta_\Lambda=\frac{1}{\pi^2}\int_0^{2\pi} h_\Phi(\theta)
d\theta. \]
\end{lemma}

\begin{proof} The lemma is a special case of Theorem 3 in Chapter IV
of \cite{Le56}. We note in passing that this is the point where the
regular distribution of $\Lambda$ is crucial.
\end{proof}

\begin{lemma}
\label{indicator2} Let $\Phi$ be a function in $\cF$ such that
whenever $G$ is a nontrivial entire function with nonempty zero set,
the function $G\Phi$ does not belong to $\cF$. Then
\[ \int_{0}^{2\pi} h_\Phi (\theta) d\theta \ge 2\pi .\]
\end{lemma}

\begin{proof}
On the assumption of the lemma, we will first prove the following
basic relation: \beq \label{nodrops}
\sup_{\theta\in[\theta_0-\pi/4,\theta_0+\pi/4)}
h_{\Phi}(\theta)=\frac{\pi}{2} \eeq for every $\theta_0 $ in
$[0,2\pi)$. We argue again by contradiction and assume that
\eqref{nodrops} fails, say for $\theta_0=0$. In other words, we
assume that
\beq
\label{epsilon} \epsilon = \frac \pi 2 - \sup_{|\theta|\leq \pi/4
} h_{\Phi}(\theta) >0. \eeq We will construct a function $\Psi$
in $\cF$ whose zero set is larger than $\Lambda
\setminus\{\lambda'\}$. The existence of $\Psi$ will contradict
the assumption of the lemma.

Consider the Mittag-Leffler function \beq \label{mlf}
E_{\frac{1}{2}}(z)=\sum_{k=0}^\infty \frac {z^k}{\Gamma(1+k/2)}. \eeq We
refer to Section 18.1 of \cite{Erd} for a detailed exposition of the
properties of this function. What we need, is that $E_{\frac{1}{2}}$ is an
entire function of order 2 with infinitely many zeros, and also that
\beq
\label{mlfestimate}
E_{\frac{1}{2}}(z)=\begin{cases} 2e^{z^2}+ O\big (|z|^{-1}\big ), & |\arg z| \leq \frac \pi 4 \\
O\big (|z|^{-1}\big ), & \frac \pi 4 \leq |\arg z| \leq \pi
\end{cases}
\eeq
when $|z|\to \infty$; see relations (7) and (9) in Section 18.1 of
\cite{Erd}.

We now introduce the following function:
\[
\Psi(z) = \Phi(z) E_{\frac{1}{2}}\bigl(z (\epsilon/2)^{1/2}\bigr).
\]
It follows from \eqref{mlfestimate} that
\[
|\Psi(z )|\leq C |\Phi(z )|, \ \ \ \frac \pi 4 \leq |\arg z|
\leq \pi,
\]
for some constant $C$, and since $\Phi \in \cF$, we obtain
\[
\int_{\frac \pi 4 \leq |\arg z| \leq \pi} |\Psi(z)|^2 e^{-\pi
|z|^2} dm(z) < \infty.
\]
On the other hand, using \eqref{mlfestimate} and an estimate based
on the Phragm\'{e}n--Lindel\"{o}f principle (see Theorem~28 in
Chapter I of \cite{Le56}), we get
\[
|\Psi(z)|\leq Ce^{(\frac \pi 2 - \frac \epsilon 3) |z|^2}, \ \ \
|\arg z| \leq \frac \pi 4,
\]
with $C$ a positive constant. We conclude that $\Psi$ belongs to
$\cF$, which is the desired contradiction.

To obtain the conclusion of the lemma from \eqref{nodrops}, we
will use that the indicator function $h_\Phi(\theta)$ is a
2-trigonometrically convex function. A consequence of this
property is that if $h_\Phi$ has a local maximum at $\theta_0$,
then
\begin{equation}
\label{maximum} h_\Phi(\theta)\geq h_\Phi(\theta_0)\cos
2(\theta-\theta_0), \ \ \ |\theta-\theta_0|\leq\frac{\pi}{4}.
\end{equation}
We refer to Chapter 1 of \cite{Le56} or Lecture 8 in \cite{Le96}
for definitions and proofs.

%It follows from Lemma \ref{nodrops} that $\max h_\Phi(\theta) = \pi/2$ and also that there exists
To finish the proof, we need the following auxiliary construction.
For each collection of directions $\theta_1, \theta_2, \ldots \ ,
\theta_n, \theta_{n+1}$ such that $0\leq \theta_1 < \theta_2 <
\ldots \ < \theta_n < 2\pi \leq \theta_{n+1}=\theta_1+2\pi$, we
set $\theta_j'= (\theta_{j+1}-\theta_j)/2$ for $j=1,2,\ldots \ ,
n$ and $\theta_{n+1}'= \theta_1'+2\pi$, and we then define the
function $H(\theta; \theta_1, \theta_2, \ldots \ , \theta_n)$ by
declaring that
\[
H(\theta; \theta_1, \theta_2, \ldots \ , \theta_n)=
\frac \pi 2 \cos 2 (\theta_{j+1}-\theta), \medskip \ \theta_j' \leq \theta < \theta_{j+1}'
\]
for $j=1,2,\ldots \ , n$. If we also require that
$\theta_{j+1}-\theta_j \leq \pi/2$, then
\beq
\label{level}
\int_0^{2\pi} H(\theta; \theta_1, \theta_2, \ldots \ , \theta_n) d\theta
\geq \int_0^{2\pi}H(\theta; 0,\pi/2, \pi, 3\pi/2)d\theta,%= 2\pi,
\eeq
as can be verified by a comparison of the level sets of the two
integrands. It follows from \eqref{nodrops} that
$\max h_\Phi(\theta) = \pi/2$ and also that there is a finite set $\{\theta_j\}_1^n$
of global maxima of $h_\Phi$ such that $|\theta_j-\theta_{j+1}|\le \pi/2$ for
$j=1,2,..,,n$; here $\theta_{n+1}=\theta_1+2\pi$ as before. Using
\eqref{maximum} and \eqref{level}, we get
%To see this we take the $4$ longest intervals of the form
%$[\theta_j,\theta_{j+1}]$. It is easy to check that the middle
%points $\theta_j'$ of this intervals are local minimum of $H(\theta;
%\theta_1, \theta_2, \ldots \ , \theta_n)$ and also they are the $4$
%smaller values of this function. We consider the eight sub intervals
%of the form $[\theta_j,\theta_j']$ and $[\theta_j',\theta_{j+1}]$ In
%any of this intervals $H(\theta-\theta_j; \theta_1, \theta_2, \ldots
%\ , \theta_n)=H(\theta-k\pi/2; 0,\pi/2, \pi, 3\pi/2)$ for
%$k=0,\dotsc,3$. Then, we can see using translations that we have a
%one to one correspondence of subintervals where we have equality,
%and in the rest of $[0,2\pi]$ it is obvious that the values of
%$H(\theta-k\pi/2; 0,\pi/2, \pi, 3\pi/2)d\theta$ will be less than
%the remaining values of $H(\theta-\theta_j; \theta_1, \theta_2,
%\ldots \ , \theta_n)$.
%, while
%the right equality comes from a straight calculation.
\[
\int_0^{2\pi} h_\Phi(\theta) d\theta \geq
\int_0^{2\pi} H(\theta; \theta_1, \theta_2, \ldots \ , \theta_n) d\theta
\geq \int_0^{2\pi}H(\theta; 0,\pi/2, \pi, 3\pi/2)d\theta= 2\pi,
\]
where the last equality follows by a simple calculation.
\end{proof}

\end{document}